\documentclass[12pt]{article}
\usepackage{color,latexsym,amsthm,amsmath,amssymb,path,enumitem}

\newcommand{\ignore}[1]{}

\makeatletter
\def\th@plain{%
\thm@notefont{}
\itshape 
}
\def\th@definition{%
\thm@notefont{}
\normalfont 
}
\makeatother

\newtheorem{theorem}{Theorem}[section]

\newtheorem{corollary}[theorem]{Corollary}

\newtheorem{claim}[theorem]{Claim}

\usepackage{dsfont}
\usepackage{comment}

\newcommand{\WW}{\ensuremath{\mathbb S}}

\newcommand{\RR}{\ensuremath{\mathbb R}}
\newcommand{\R}{\ensuremath{\mathbb R}}
\newcommand{\CC}{\ensuremath{\mathbb C}}

\newcommand{\DD}{\ensuremath{\mathbb D}}

\newcommand{\lines}{\mathcal L}
\newcommand{\circs}{\mathcal C}

\newcommand{\pts}{\mathcal P}

\def\eps{{\varepsilon}}

\newcommand{\parag}[1]{\vspace{2mm}

\noindent{\bf #1}}

\title{Distance problems for planar hypercomplex numbers\thanks{This research project was done as part of the 2019 CUNY Combinatorics REU, supported by
NSF awards DMS-1802059 and DMS-1851420.}}
\author{David FitzPatrick\thanks{Princeton University, Princeton, NJ, USA. \textit{davidbf@princeton.edu}}}
\begin{document}
\pagenumbering{arabic}
\date{}
\maketitle

\begin{abstract}
We study the unit distance and distinct distances problems over the planar hypercomplex numbers: the dual numbers $\DD$ and the double numbers $\WW$.
We show that the distinct distances problem in $\WW^2$ behaves similarly to the original problem in $\R^2$.
The other three problems behave rather differently from their real analogs.
We study those three problems by introducing various notions of multiplicity of a point set.

Our analysis is based on studying the geometry of the dual plane and of the double plane.
We also rely on classical results from discrete geometry, such as the Szemer\'edi--Trotter theorem.
\end{abstract}

\section{Introduction}
The unit distance and distinct distances problems are two of the most celebrated problems in discrete geometry.
In this paper, we study those problems over the planar hypercomplex numbers: the dual numbers and the double numbers.

Dual and double numbers appear in many different fields.
For example, double numbers are used in string theory \cite{GGP96}, in signal processing \cite{MG09}, and to design algorithms for dating sites \cite{KGG12}. Dual numbers are used in kinematics \cite{Fischer00}, in the theory of schemes \cite{Hars83}, and in works studying the Erlangen program \cite{Kisil12}.
However, it seems that a combinatorial study of these numbers only started very recently in \cite{HLS12}.
The current work continues this project.

\parag{Unit distance problem.} Erd\H os \cite{erd46} suggested the unit distance problem: What is the maximum possible number of pairs of points at a distance of one from each other in a set of $n$ points in $\R^2$?
To quote from the book ``Research Problems in Discrete Geometry'' \cite{BMP06}, this is ``possibly the best known (and simplest to explain) problem in combinatorial geometry.''

By taking a set of $n$ equally-spaced points on a line, we get $n - 1$ unit distances. When Erd\H os introduced the problem, he derived an asymptotically stronger lower bound of $\Omega(n^{1+c/\log\log n})$ (for some constant $c$) and an upper bound of $O(n^{3/2})$.
Although this is a central problem in discrete geometry, in the decades that have passed, the lower bound has never been improved and the upper bound has been improved only once.
In 1984, Spencer, Szemer\'edi, and Trotter \cite{SST84} derived the bound $O(n^{4/3})$.

Many variants of the unit distance problem have also been studied.
For example, the unit distance problem has been studied in $\RR^d$ \cite{FPSSZ17,KMSS12,Zahl11}, in $\CC^2$ \cite{ST12}, and using other distance norms \cite{mat11,Valtr05}.

\parag{Distinct distances problem.} Erd\H{o}s \cite{erd46} also posed the distinct distances problem: What is the minimum number of distinct distances determined by pairs of points from a set of $n$ points in $\R^2$?
Erd\H{o}s showed that a $\sqrt{n}\times\sqrt{n}$ section of the integer lattice determines $\Theta \left( \frac{n}{\sqrt{\log n}} \right)$ distinct distances and conjectured that this was asymptotically the fewest possible.
To date, the upper bound has never been asymptotically improved.
The lower bound, however, was steadily improved over the years.
Recently, Guth and Katz \cite{GK15} introduced novel polynomial methods to derive the bound $\Omega \left( \frac n {\log n} \right)$, which matches Erd\H{o}s's conjecture up to a factor of $\sqrt{\log n}$.

The distinct distances problem also has a large number of variants, and some of the main variants remain wide open.
For example, the problem has been studied in higher dimensions \cite{SV08}, in finite fields \cite{CE12,IKP19,LP19,MRS19}, and with bipartite distances \cite{M19}.
For more information, see this book about the problem \cite{GIS11} and a survey of open distinct distances problems \cite{Sheffer14}.

\parag{Dual numbers.}
Let $\DD$ be the set of dual numbers.
This is the two-dimensional unital associative $\R$-algebra obtained by adding to $\R$ the additional element $\eps$ and the rule $\eps^2=0$.
There is a unique way to write any dual number in the standard form $x+y\eps$, where $x,y\in \RR$.
Following the terminology of the complex numbers, we refer to $x$ as the \emph{real part} and to $y$ as the \emph{imaginary part}.

We define the \emph{dual plane} $\DD^2$ to be the set of all points $p = (x + y\eps, z + w\eps)$. We define the \emph{real part} of a point $p \in \DD^2$ to be $(x, z) \in \R^2$ and the \emph{imaginary part} to be $(y, w) \in \R^2$.
For brevity, we use the notation $[x, y]$ for the dual number $x + y\epsilon \in \DD$ and $[x, y, z, w]$ for the point $(x + y\eps, z + w\eps) \in \DD^2$.
We sometimes think of $\DD^2$ as $\RR^4$, defined by the coordinates $x, y, z, w$.

When working in the dual plane $\DD^2$, we define the \emph{imaginary plane} associated with a real point $p =(x, z) \in \R^2$ to be the set of points in $\DD^2$ of the form $[x, y, z, w] \in \DD^2$.
We denote this imaginary plane by $H_p$.
In other words, $H_p$ is the set of points of $\DD^2$ that have $p$ as their real part.
Note that $H_p$ is indeed a plane when we think of $\DD^2$ as $\RR^4$.
For a line $\ell \in \R^2$, we refer to the set of points $[x, z, y, w] \in H_p$ satisfying $(y, w) \in \ell$ as the copy of $\ell$ in the imaginary plane $H_p$.

We consider the Euclidean-style distance function $\rho_{\DD^2} : \DD^2 \times \DD^2 \to \DD$, defined as
\begin{align}
\rho_{\DD^2}([x, y, z, w], [x', y', z', w']) &= \rho_{\DD^2}((x + y\eps, z + w\eps), (x' + y'\eps, z' + w'\eps)) \nonumber\\
&= ((x + y \eps) - (x' + y' \eps))^2 + ((z + w\eps) - (z' + w'\eps))^2 \nonumber\\
&= (\Delta x+\Delta y \cdot \eps)^2+(\Delta z+\Delta w \cdot \eps)^2 \nonumber\\
&= (\Delta x)^2+ (\Delta z)^2 + 2(\Delta x \Delta y + \Delta z \Delta w)\eps \nonumber\\
&= [(\Delta x)^2+ (\Delta z)^2, 2(\Delta x \Delta y + \Delta z \Delta w)]. \label{eq:distDual}
\end{align}

Functions defined in this way are a common means of measuring distances in planes over finite fields (for example, see \cite{BKT04,SdZ17}), in $\CC^2$ (see \cite{ST12}), and more. They are always symmetric.
However, $\rho_{\DD^2}(\cdot)$ is not quite a valid metric, since $\rho_{\DD^2}(p, q) =0$ does not imply $p = q$ (where $p,q\in \DD^2$).
Instead, $\rho_{\DD^2}(p, q) =0$ if and only if $p$ and $q$ have the same real part.

Consider the unit distance problem in $\DD^2$.
That is, we are interested in the maximum number of pairs satisfying $[(\Delta x)^2 + (\Delta z)^2, 2(\Delta x\Delta y + \Delta z\Delta w)] = [1, 0]$.
A previous work studying combinatorial properties of hypercomplex numbers \cite{HLS12} noted that a set of dual numbers exhibits degenerate behavior when it contains many numbers with the same real part.
This is also the case when studying unit distances in $\DD^2$.
For example, consider the set
\[ \pts = \left\{[a, 0, 0, b]\ :\ a \in \{0, 1\},\ b \in\left\{1, \dots, n \right\}\right\}. \]
Note that $\pts$ is a set of $\Theta(n)$ points in $\DD^2$ that spans $\Theta(n^2)$ unit distances.
This happens because the elements of $\pts$ have only two distinct real parts: $(0, 0)$ and $(1, 0)$.
Thus, the unit distance problem is trivial in $\DD^2$ if we allow a constant portion of the points to have the same real part.

Consider a set $\pts\subset \DD^2$.
Building on the analysis in \cite{HLS12}, we define the \emph{multiplicity} of $\pts$ as the largest integer $k$ such that there exist $k$ points of $\pts$ with the same real part.
In other words, there exists an imaginary plane $H_p$ that contains $k$ points of $\pts$.

For any $0 \leq \lambda \leq 1$, we can adapt the above construction to obtain a set with multiplicity $\Theta(n^\lambda)$ that spans $\Theta(n^{1 + \lambda})$ unit distances.
Consider the set
\begin{equation} \label{eq:MultConstruction}
\pts = \left\{[a, 0, 0, b]\ :\ a \in \left\{1, ..., n^{1 - \lambda}\right\},\ b \in \left\{1, ..., n^{\lambda}\right\}\right\}.
\end{equation}
We have that $|\pts|=n$ and that the number of unit distance spanned by $\pts$ is $\Theta(n^{1 + \lambda})$.
Indeed, two points of $\pts$ span a unit distance if and only if their $a$ values differ by one.
This implies that every point of $\pts$ spans a unit distance with $\Theta(n^\lambda)$ other points.

Next note that every point set in $\RR^2$ can be associated to a point set in $\DD^2$ that contains the same number of points, spans the same number of unit distances, and has multiplicity 1 (the minimum possible multiplicity for a nonempty point set in $\DD^2$).
Indeed, this can be done by replacing every point $(p_x,p_y)\in \R^2$ with $[p_x, 0, p_y, 0] \in \DD^2$.

Therefore, no matter how small $\lambda$ is, if we showed that every set of $n$ points in $\DD^2$ with multiplicity at most $n^{\lambda}$ spans $O(f(n))$ unit distances, then we could conclude that, in particular, every set of $n$ points in $\RR^2$ spans $O(f(n))$ distances.
Thus, regardless of the multiplicity of a point set $\pts \subset \DD^2$, we cannot expect to obtain an upper bound stronger than $O(n^{4/3})$ on the number of unit distances spanned by $\pts$.
A stronger bound would improve on the $O(n^{4/3})$ bound for unit distances in $\RR^2$, which no one has been able to do since that bound was introduced in the early 1980s.

To recap, the best upper bound we can hope to obtain for the unit distance problem in $\DD^2$ as a function of the cardinality $n$ and multiplicity $n^{\lambda}$ of the point set is $O(n^{1 + \lambda} + n^{4/3})$.
In the current work we derive this bound, up to polylogarithmic factors.

\begin{theorem} \label{th:UnitDual}
Let $\pts$ be a set of $n$ points in $\DD^2$ with multiplicity $n^\lambda$, for some $0\le \lambda \le 1$.
Then the number of unit distances spanned by $\pts$ is
\[ O \left( n^{1+\lambda} \cdot \log^4 n+n^{4/3} \cdot \log^2 n\right). \]
\end{theorem}
In Section \ref{sec:Dual}, after proving Theorem \ref{th:UnitDual}, we show that the theorem still holds when replacing the unit distance with any other nonzero distance in $\DD$.

We now move to the distinct distances problem in $\DD^2$.
For a point set $\pts$, in any space, we denote the number of distinct distances spanned by $\pts$ by $D(\pts)$.
As before, degenerate cases arise for this problem when we allow many points with the same real part.
For example, when the real parts of all the points of a set are identical, the only distance is 0.

Consider the set $\pts$ from \eqref{eq:MultConstruction}.
This is a set of $n$ points with multiplicity $n^\lambda$ that satisfies $D(\pts)=n^{1 - \lambda}$.
Indeed, the distances spanned by $\pts$ are exactly the integers $a^2$ for $0\le a \le n^{1 - \lambda}-1$.

This suggests that we should once again try to derive a bound that is a function of the multiplicity $n^{\lambda}$ of a point set. However, for this problem, we will actually be able to obtain a bound in terms of a finer notion of multiplicity.
Recall that the multiplicity of a set $\pts\subset \DD^2$ can be thought of as the largest number of points of $\pts$ in an imaginary plane $H_p$.
We define the \emph{secondary multiplicity} of $\pts$ to be the largest $k$ such that there exists a line $\ell$ in some imaginary plane $H_p$ that is incident to $k$ points of $\pts$.
By definition, the secondary multiplicity cannot be larger than the multiplicity, but it may be significantly smaller. It is easy to see that the secondary multiplicity of the set $\pts$ in the above example is also $n^{\lambda}$, since all the points in a given imaginary plane lie on a common line. Thus, the best lower bound we can possibly hope to obtain for the distinct distances problem in $\DD^2$ in terms of the secondary multiplicity $n^{\nu}$ of the point set is $\Omega(n^{1 - \nu})$. We obtain this bound up to polylogarithmic factors:

\begin{theorem} \label{th:DistinctDual}
Let $\pts$ be a set of $n$ points in $\DD^2$ with secondary multiplicity $n^{\nu}$.
Assume that the points of $\pts$ do not all have the same real part. Then
\[ D(\pts)=\Omega \left( n^{1-\nu}\log^{-2} n \right). \]
\end{theorem}

\parag{Double numbers.}
Let $\WW$ be the set of double numbers (also called the split-complex numbers and the hyperbolic numbers).
This is the two-dimensional unital associative $\R$-algebra obtained by adding to $\RR$ the additional element $j$ and the rule $j^2= 1$.
There is a unique way to write any double number in the standard form $X+Yj$, where $X,Y\in \RR$. (We use capital letters because we will soon switch to different coordinates, which we will denote by lowercase letters.)

The \emph{double plane} $\WW^2$ is the set of all pairs of the form $p = (X + Y j, Z + Wj)$.

As in the dual case, we consider the Euclidean-style distance function $\rho_{\WW^2} : \WW^2 \times \WW^2 \to \WW$, defined as
\begin{align}
& \rho_{\WW^2}((X +Yj, Z + Wj), (X' + Y'j, Z' + W'j)) \nonumber \\
&= ((X + Y j) - (X' + Y' j))^2 + ((Z + Wj) - (Z' + W'j))^2 \nonumber \\
&= (\Delta X+\Delta Y j)^2+(\Delta Z+\Delta W j)^2 \nonumber \\
&= \Delta X^2+\Delta Y^2 + \Delta Z^2 + \Delta W^2 + 2(\Delta X \Delta Y + \Delta Z \Delta W)j. \label{eq:DoubleOriginalDist}
\end{align}

We now introduce a change of coordinates in $\WW$ that significantly simplifies the distance $\rho_{\WW^2}(\cdot)$.
Given an arbitrary point $p = X + Yj \in \WW$, define the coordinates
\[x = X + Y, \qquad y = X - Y.\]
We represent the point $p$ with the notation $\langle x, y\rangle $, where $x$ and $y$ are the new coordinates defined above.
In other words, we have that
\begin{equation}
\langle x, y\rangle = X+Yj =(x+y)/2 + (x-y)j/2. \label{eq:CoordinateChangeNumber}
\end{equation}

Given a point $q = (X + Yj, Z + Wj) \in \WW^2$, we represent $q$ with the notation $\langle x, y, z, w\rangle $, where
\begin{align}
x &= X + Y, \quad y = X - Y, \quad z = Z + W, \quad w = Z - W, \text{ and } \label{eq:CoordinateChange}\\[2mm]
X &= \frac{1}{2} (x + y), \quad Y= \frac{1}{2} (x- y), \quad Z = \frac{1}{2} (z + w), \quad W = \frac{1}{2} (z - w). \nonumber
\end{align}

Combining the new coordinates with \eqref{eq:DoubleOriginalDist} gives
\begin{align}
\rho_{\WW^2}&(\langle x, y, z, w\rangle , \langle x', y', z', w'\rangle ) \nonumber \\
&= \Delta X^2+\Delta Y^2 + \Delta Z^2 + \Delta W^2 + 2(\Delta X \Delta Y + \Delta Z \Delta W)j \nonumber \\
&= \frac 1 4 (\Delta x + \Delta y)^2 + \frac 1 4 (\Delta x - \Delta y)^2 +
\frac 1 4 (\Delta z + \Delta w)^2 + \frac 1 4 (\Delta z - \Delta w)^2 \nonumber \\
&\hspace{45mm} + \frac 1 2 \big((\Delta x + \Delta y)(\Delta x - \Delta y) + (\Delta z + \Delta w)(\Delta z - \Delta w)\big)j \nonumber \\
& =\frac{1}{2}\left((\Delta x)^2 + (\Delta y)^2 + (\Delta z)^2 + (\Delta w)^2\right) + \frac{1}{2}\left((\Delta x)^2 - (\Delta y)^2 + (\Delta z)^2 - (\Delta w)^2\right)j \nonumber \\
&=\langle (\Delta x)^2 + (\Delta z)^2, (\Delta y)^2 + (\Delta w)^2\rangle. \label{eq:DoubleDist}
\end{align}

We can think of the distance function $\rho_{\WW^2}$ as the Cartesian product of the squares of two standard Euclidean metrics on $\R^2$. Note that this distance function is symmetric and that for $p, q \in \WW^2$, we have $\rho_{\WW^2}(p,q) = 0 = \langle 0, 0\rangle$ if and only if $p=q$.

Unlike in the real and dual cases, the maximum number of repetitions of a distance $d$ in a set $\pts\subset \WW^2$ depends on the value of $d$.
Consider a distance $d = \langle d_1, d_2 \rangle \neq \langle 0, 0 \rangle$.
If both $d_1$ and $d_2$ are positive, then we say that $d$ is a \emph{type A distance}.
If $d_1 = 0$ and $d_2 > 0$, we say that $d$ is a \emph{type B distance}. If $d_1 > 0$ and $d_2 = 0$, we say that $d$ is a \emph{type C distance}.
Since $d_1$ and $d_2$ are non-negative, every nonzero distance is of one of these three types.

In \cite{HLS12}, it was observed that a set $S \subset \WW$ exhibits a degenerate behavior when many elements $X + Yj$ have the same value of $X + Y$ or the same value of $X - Y$.
After the coordinate change in \eqref{eq:CoordinateChange}, this condition changes to many elements $\langle x,y \rangle$ having the same value of $x$ or the same value of $y$.

We define the \emph{real part} of a point $p = \langle x, y, z, w\rangle \in \WW^2$ to be $(x, z) \in \R^2$.
We define the \emph{imaginary part} of $p$ to be $(y, w) \in \R^2$.
We define the \emph{real multiplicity} of a set $\pts\subset \WW^2$ to be the largest integer $k_1$ such that there exist $k_1$ points in $\pts$ with the same real part. We define the \emph{imaginary multiplicity} of $\pts$ to be the largest integer $k_2$ such that there exist $k_2$ points in $\pts$ with the same imaginary part. Finally, we define the \emph{minimal multiplicity} of $\pts$ to be the minimum of its real and imaginary multiplicities. Each of these three notions of multiplicity turns out to be useful for studying repeated distances of one of the three types defined above.

\begin{theorem} \label{th:UnitDouble}
Let $\pts$ be a set of $n$ points in $\WW^2$. Let $d \in \WW$ be a non-zero distance. \\
(a) If $d$ is of type A and $\pts$ has minimal multiplicity $n^{\lambda}$ (where $0\le \lambda \le 1$), then the number of times $d$ is spanned by $\pts$ is
\[ O \left( \left(n^{4/3 + \lambda/3} + n^{1 + \lambda} \right) \log^2 n\right). \]
(b) If $d$ is of type B and $\pts$ has real multiplicity $n^{\lambda}$, then the number of times $d$ is spanned by $\pts$ is $\displaystyle O \left( n^{1 + \lambda/3}\right)$. \\
(c) If $d$ is of type C and $\pts$ has imaginary multiplicity $n^{\lambda}$, then the number of times $d$ is spanned by $\pts$ is $\displaystyle O \left( n^{1 + \lambda/3}\right)$.
\end{theorem}

In Section \ref{sec:Double}, after proving Theorem \ref{th:UnitDouble}, we provide lower bound constructions for the problem.
These show that Theorem \ref{th:UnitDouble} cannot be improved without also improving the upper bound for the unit distance problem in $\RR^2$ (except possibly for the $\log^2 n$ factor in part (a)).

Unlike the three problems considered above, it turns out that the distinct distances problem in $\WW^2$ behaves similarly to the distinct distances problem in $\RR^2$.
In this case, we get a bound that is tight up to polylogarithmic factors and does not depend on any multiplicity.

\begin{theorem} \label{th:DistinctDouble}
Every set of $n$ points in $\WW^2$ determines $\Omega \left( \frac{n}{\log^3 n} \right)$ distinct distances.
\end{theorem}

In Section \ref{sec:Double}, we also describe a set $\pts\subset \WW^2$ of $n$ points with $D(\pts) = \Theta(n/\log n)$.
This leaves a gap of $\log^2 n$ for the distinct distances problem in $\WW^2$.

\section{Preliminaries}
Let $\pts$ be a set of points and let $\lines$ be a set of lines, both in $\RR^2$.
A point--line pair $(p,\ell)\in \pts\times\lines$ is an \emph{incidence} if $p$ is on $\ell$.
We denote by $I(\pts,\lines)$ the number of incidences in $\pts\times\lines$.

\begin{theorem}[Szemer\'edi--Trotter \cite{ST83}] \label{th:ST83}
Let $\pts$ be a set of $m$ points and let $\lines$ be a set of $n$ lines, both in $\RR^2$.
Then
\[ I(\pts,\lines)=O\left(m^{2/3}n^{2/3}+m+n\right). \]
\end{theorem}

Given a point set $\pts\subset \RR^2$ and an integer $r$, we say that a line $\ell$ is $r$-\emph{rich} if $\ell$ is incident to at least $r$ points of $\pts$.
We similarly define $r$-rich circles, planes, and other objects.
The following is known as a dual form of Theorem \ref{th:ST83}, in the sense that each result can be easily derived from the other.

\begin{corollary} \label{co:DualST}
Let $\pts$ be a set of $n$ points in $\RR^2$ and let $r\ge2$ be an integer.
Then the number of $r$-rich lines is
\[ O\left(\frac{n^2}{r^3} + \frac{n}{r}\right).\]
\end{corollary}

Note that the term $n^2r^{-3}$ dominates the bound of Corollary \ref{co:DualST} when $r =O(n^{1/2})$. The term $nr^{-1}$ dominates the bound when $r =\Omega(n^{1/2})$. And note that when $r = 1$, there are infinitely many $r$-rich lines for any non-empty set $\pts$.

We also rely on a bound for incidences with unit circles (see for example \cite[Theorem 8]{Szek93}).

\begin{theorem} \label{th:Szek93}
Let $\pts$ be a set of $m$ points and let $\circs$ be a set of $n$ unit circles, both in $\RR^2$.
Then
\[ I(\pts,\circs)=O\left(m^{2/3}n^{2/3}+m+n\right). \]
\end{theorem}

Theorem \ref{th:Szek93} implies the current best bound $O(n^{4/3})$ for the unit distance problem in $\RR^2$.
We also rely on bounds for the distinct distances problem in $\R^2$.
The following result is from the seminal work of Guth and Katz \cite{GK15}.

\begin{theorem} \label{th:GK15 2}
Every set of $n$ points in $\RR^2$ determines $\Omega \left( \frac{n}{\log n} \right)$ distinct distances.
\end{theorem}

In a \emph{bipartite} distinct distances problem, we are interested in the minimum number of distinct distances between two (not necessarily disjoint) sets $\pts_1, \pts_2 \subset \R^2$.
That is, we consider the distances between pairs of points $x$ and $y$ with $(x, y) \in \pts_1\times \pts_2$.
We denote the number of such distinct distances by $D(\pts_1,\pts_2)$.
The following bound was derived in \cite{M19}.

\begin{theorem} \label{th:M19}
Let $\pts_1$ be a set of $n$ points and let $\pts_2$ be a set of $m$ points, both in $\RR^2$, such that $2 \leq m \leq n$.
Then $\displaystyle D(\pts_1,\pts_2) = \Omega \left( \sqrt{mn}/ \log n \right)$.\end{theorem}

\section{Dual numbers} \label{sec:Dual}

We now prove our results in the dual plane $\DD^2$.
We restate each result before proving it.
\vspace{2mm}

\noindent {\bf Theorem \ref{th:UnitDual}.}
\emph{Let $\pts$ be a set of $n$ points in $\DD^2$ with multiplicity $n^\lambda$, for some $0\le \lambda \le 1$.
Then the number of unit distances spanned by $\pts$ is}
\[ O \left( n^{1+\lambda} \cdot \log^4 n+n^{4/3} \cdot \log^2 n\right). \]

\begin{proof}
Consider points $p=[x,y,z,w]$ and $q=[x', y', z', w']$ in $\DD^2$ that span a unit distance.
In other words, we have
\[ [1, 0] = \rho_{\DD^2}(p, q) =\left[(x-x')^2+(z-z')^2, 2((x-x')(y-y')+(z-z')(w-w'))\right]. \]
Splitting this into real and imaginary parts, we get
\begin{align}
(x-x')^2+(z-z')^2 &=1, \label{eq:UnitCirc1} \\[2mm]
(x-x')(y-y') + (z-z')(w-w') &= 0. \label{eq:UnitCirc2}
\end{align}

By \eqref{eq:UnitCirc1}, the real parts $(x, z)$ and $(x', z')$ span a unit distance in $\R^2$.
By \eqref{eq:UnitCirc2}, the vector $(y - y', w - w')$ is perpendicular to the vector $(x - x', z - z')$.
Let $\ell(p, q)$ be the line in $\R^2$ that is incident to $(y, w)$ and has direction orthogonal to the vector $(x - x', z- z')$. By definition, $\ell(p, q)$ is incident to both $(y, w)$ and $(y', w')$.
Thus, the copy of $\ell(p, q)$ in the imaginary plane $H_{(x, z)}$ is incident to $p$ and the copy of $\ell(p, q)$ in $H_{(x', z')}$ is incident to $q$.

\parag{Dyadically decomposing the problem.}
For any integers $0 \leq \alpha, \beta, \gamma, \delta \leq \log n$, define $I(\alpha, \beta, \gamma, \delta)$ to be the number of pairs of points $p = [x, y, z, w]$ and $q = [x', y', z', w']$ in $\pts$ that define a unit distance and satisfy the following properties:
\begin{itemize}
\item The number of points of $\pts$ that have the same real part as $p$ is at least $2^{\alpha}$ and smaller than $2^{\alpha + 1}$. Equivalently, this is the number of points of $\pts$ in $H_{(x,z)}$.
\item The number of points of $\pts$ that have the same real part as $q$ is at least $2^{\beta}$ and smaller than $2^{\beta + 1}$.
\item In the imaginary plane $H_{(x, z)}$, the number of points of $\pts$ that lie on the copy of $\ell(p, q)$ is at least $2^\gamma$ and smaller than $2^{\gamma + 1}$.
\item In the imaginary plane $H_{(x', z')}$, the number of points of $\pts$ that lie on the copy of $\ell(p, q)$ is at least $2^\delta$ and smaller than $2^{\delta + 1}$. 
\end{itemize}

There are $O(\log^4 n)$ possible values for the tuple $(\alpha, \beta, \gamma, \delta)$.
Each unit distance that is spanned by $\pts$ contributes to $I(\alpha,\beta,\gamma,\delta)$ for exactly one of these tuples. Moreover, $O(\log^2 n)$ of the tuples satisfy $\gamma = \delta = 0$.
Thus, the number of unit distances spanned by $\pts$ is at most $O(\log^4 n)$ times the maximum possible size of an $I(\alpha,\beta,\gamma,\delta)$ that does not satisfy $\gamma = \delta = 0$, plus $O(\log^2 n)$ times the maximum possible size of an $I(\alpha,\beta,\gamma,\delta)$ that does.

Since $\pts$ has multiplicity $n^\lambda$, we may assume that $2^\alpha \le n^\lambda$ and $2^\beta \le n^\lambda$.
By definition, we also have that $\gamma \leq \alpha$ and $\delta \leq \beta$.
We fix values of $\alpha, \beta, \gamma,\delta$ with these properties.
From now on, we consider only unit distances that contribute to this specific $I(\alpha, \beta, \gamma, \delta)$.
We partition the rest of the proof into five cases, according to the values of $\alpha, \beta, \gamma,\delta$.

\parag{Case 1:} $\alpha /2 \leq \gamma \leq \alpha$.

Suppose that the points $p$ and $q$ define a unit distance. There are $O(n2^{-\alpha})$ possible values for the real part $(x, z)$ of $p$.
Indeed, every such real part exhausts $\Theta(2^\alpha)$ of the $n$ points of $\pts$.
Fix one such real part $(x, z)$.
Recall that the copy of $\ell(p, q)$ in $H_{(x, z)}$ must be $\Theta(2^{\gamma})$-rich.
We apply Corollary \ref{co:DualST} to obtain an upper bound on the number of such lines.
Since $H_{(x, z)}$ contains fewer than $2^{\alpha+1}$ points and we are in the case of $\gamma \ge \alpha/2$, the corollary implies that the number of distinct lines that are candidates for $\ell(p, q)$ is $O(2^{\alpha - \gamma})$.

Fix a line $\ell \in \R^2$ with the above properties. Recall that $(x', z')$ must be such that $(x - x', z - z')$ is a vector of Euclidean norm 1 orthogonal to $\ell$. Thus, there are at most two possible values for $(x', z')$. In addition, since the copy of $\ell(p, q)$ in $H_{(x, z)}$ must be incident to the imaginary part $(y, w)$ of $p$, the remaining possibilities for $(y, w)$ are the imaginary parts of the points of $\pts$ lying on the copy of $\ell$ in $H_{(x, z)}$. By construction, there are $\Theta(2^\gamma)$ such imaginary parts.

Fix choices of $(x', z')$ and $(y, w)$ with the above properties. Because the copy of $\ell(p, q)$ in $H_{(x', z')}$ must be incident to the imaginary part $(y', w')$ of $q$, the remaining possibilities for $(y', w')$ are the points of $\pts$ lying on the copy of $\ell$ in $H_{(x', z')}$, of which there are $\Theta(2^\delta)$.

Combining the above implies that, in this case, we have
\[ I(\alpha, \beta, \gamma, \delta) =O(n2^{-\alpha}) \cdot O(2^{\alpha - \gamma}) \cdot 2 \cdot \Theta(2^{\gamma}) \cdot \Theta(2^{\delta}) = O\left( n2^{\delta} \right) = O\left( n2^{\beta}\right) =O\left( n^{1 + \lambda}\right).\]

\parag{Case 2:} $\beta /2 \leq \delta \leq \beta$.

This case is symmetric to Case 1.

\parag{Case 3:} $1 \leq \gamma < \alpha /2$ and $\delta \leq \gamma$.

We follow the analysis of Case 1.
In that case, we relied on the assumption $\gamma \geq \alpha/2$ only when applying Corollary \ref{co:DualST}.
In the current case, the corollary implies that the number of $2^\gamma$-rich lines in an imaginary plane $H_{(x,z)}$ is $O(2^{2\alpha - 3\gamma})$.
We do not change any other part of the analysis of Case 1.
This leads to
\begin{align*}
I(\alpha, \beta, \gamma, \delta) =O(n2^{-\alpha}) \cdot O(2^{2\alpha - 3\gamma}) \cdot 2 \cdot \Theta(2^{\gamma}) \cdot \Theta(2^{\delta}) &= O\left( n2^{\alpha+\delta-2\gamma}\right) \\[2mm]
&=O\left(n2^\alpha\right) = O\left(n^{1+\lambda}\right) .
\end{align*}
Here we used the assumption $\delta \leq \gamma$ to conclude that $\delta-2\gamma<0$.

\parag{Case 4:} $1 \leq \delta < \frac \beta 2$ and $\gamma \leq \delta$.

This case is symmetric to Case 3.

\parag{Case 5:} $\gamma = \delta = 0$.

Note that the case of $\gamma = \delta = 0$ is the only one not covered by Cases 1--4.
Indeed, by Cases 1 and 2, we may assume that $\gamma<\alpha/2$ and $\delta<\beta/2$.
If at least one of $\gamma$ and $\delta$ is positive, then this is covered by Cases 3 and 4.

By repeating the argument at the beginning of Case 1, we get that there are $O(n2^{-\alpha})$ possible values for the real part $(x, z)$ of $p$.
By a symmetric argument, there are $O(n2^{- \beta})$ possible values for $(x', z')$.
Note that $(x,z)$ and $(x',z')$ must span a unit distance in $\RR^2$, and by Theorem \ref{th:Szek93}, the number of such pairs is
\[O\left(n^{4/3 }2^{- 2(\alpha+\beta)/3} + n2^{ - \alpha} + n2^{ - \beta}\right)\]
(We apply the theorem with the possible values for $(x, z)$ as the set of points and with a unit circle centered at each possible value of $(x', z')$).
Fix a pair $(x, z)$ and $(x', z')$ with the above properties.
The direction of the line $\ell(p, q)$ is then uniquely determined, since it must be orthogonal to $(x - x', z - z')$.
By the assumptions on $\gamma$ and $\delta$, $\ell(p, q)$ must also be such that its copies in $H_{(x, z)}$ and $H_{(x', z')}$ both contain exactly one point of $\pts$.
Recall that $H_{(x, z)}$ contains $\Theta(2^\alpha)$ points and that $H_{(x', z')}$ contains $\Theta(2^\beta)$ points.
Since parallel lines are disjoint, the number of possibilities for $\ell(p, q)$ is then $O(\min\{2^\alpha,2^\beta\})$. Fixing such an $\ell$, there is then only one possibility for each of $(y, w)$ and $(y', w')$.

Without loss of generality, we assume that $\alpha \leq \beta$.
Combining the above implies that
\[ I(\alpha, \beta, \gamma, \delta) =O\left(n^{4/3 }2^{- 2(\alpha+\beta)/3} + n2^{ - \alpha}\right) \cdot O\left(2^\alpha\right) \cdot 1 = O\left(n^{4/3 }2^{(\alpha- 2\beta)/3} + n\right)= O\left(n^{4/3}\right). \]

Combining the five cases, we conclude that $I(\alpha, \beta, \gamma, \delta) = O\left( n^{1 + \lambda}\right)$ if $\gamma = \delta = 0$ does not hold, and $I(\alpha, \beta, \gamma, \delta) = O\left( n^{4/3} \right)$ if it does.
This completes the proof of the theorem.
\end{proof}

\textit{Remark}. Recall that the secondary multiplicity of $\pts$ is the largest $k$ such that there exists a line $\ell$ in some imaginary plane $H_p$ that contains $k$ points of $\pts$.
The secondary multiplicity of $\pts$ cannot be larger than the multiplicity of $\pts$, but it can be significantly smaller.
Most of the proof of Theorem \ref{th:UnitDual} is still valid when we replace the multiplicity with the secondary multiplicity $n^\nu$.
The only cases for which we cannot easily adapt the above arguments are 3 and 4. Moreover, the argument in case 5 can be extended to handle all instances of cases 3 and 4 for which $\alpha \leq \beta$ and $\delta \leq 2\beta/ 3 - \alpha/3$, or $\beta \leq \alpha$ and $\gamma \leq 2\alpha/ 3 - \beta/3$ (at the small cost of an additional $\log^2 n$ factor on the $n^{4 / 3}$ term in our bound), leaving a narrow range of bad values.
It would be interesting to know whether it is possible to handle the remaining cases and obtain the improved bound $O((n^{1 + \nu} + n^{4/3}) \cdot \log^4 n)$.

In $\RR^2$, the unit distance problem is equivalent to finding the maximum number of times a set of $n$ points can span a distance $\delta$, for any $\delta > 0$.
Indeed, we can perform a uniform scaling of $\RR^2$, which establishes a bijection between distances 1 and distances $\delta$.
We now show that the same sort of argument works in $\DD^2$.

Let $\delta = [\delta_1,\delta_2]\in \DD$ be a valid distance between points in $\DD^2$ (i.e., $\delta$ is in the image of the metric on $\DD^2$).
By considering the distance definition in \eqref{eq:distDual}, we note that this is to say that $\delta_1\ge 0$, and if $\delta_1 = 0$, then $\delta_2 = 0$.
Recall that two points in $\DD^2$ span the distance $[0, 0]$ if and only if they have the same real part, so the repeated distance problem is not very interesting for that distance. But for any other distance, that is, any dual number of the form $[\delta_1, \delta_2]$ with $\delta_1 > 0$, the repeated distance problem for that distance is indeed equivalent to the repeated distance problem for the distance $[1, 0]$, i.e., the unit distance problem.
\begin{claim}
For any $\delta = [\delta_1, \delta_2] \in \DD$ with $\delta_1 > 0$, the unit distance problem in $\DD^2$ is equivalent to finding the maximum number of times that an $n$-point set can span the distance $\delta \in \DD$. In particular, Theorem \ref{th:UnitDual} holds for all non-zero repeated distances.
\end{claim}
\begin{proof}
Let $s\in \RR$ be non-zero.
A scaling of just the imaginary coordinates $[x, y, z, w] \mapsto [x, s y, z, s w]$ changes every distance $[d_1, d_2]$ to $[d_1, s d_2]$.
This is not difficult to see when considering the distance definition in \eqref{eq:distDual}.
Thus, the repeated distance problem is equivalent for any pair of distances of the form $[d_1, d_2]$ and $[d_1, s d_2]$.

For positive $s\in \RR$, consider the change of coordinates
\[ [x, y, z, w] \mapsto [x\cdot \sqrt{s}, y/\sqrt{s}, z\cdot \sqrt{s}, w/\sqrt{s}]. \]
By the distance definition in \eqref{eq:distDual}, this transformation changes every distance $[d_1, d_2]$ to $[sd_1, d_2]$.
Thus, the problem is also equivalent for any pair of distances of the form $[d_1, d_2]$ and $[sd_1, d_2]$.

The above allows us to show equivalence between most repeated distance problems.
By the scalings presented in the preceding paragraphs, it suffices to show that the cases of $[1,2]$ and $[1,0]$ are equivalent.

Consider two points $p=[x',y',z',w']$ and $q=[x'',y'',z'',w'']$ that span the distance $[1, 2]$.
Consider the change of coordinates
\[ [x, y, z, w] \mapsto [x, y-x, z, w- z]. \]
Since $(x' - x'')^2 + (z' - z'')^2 = 1$, we have that
\begin{align*}
&2((x'- x'')(y'- y'') + (z' - z'')(w' - w'')) = 2 \quad \text{ if and only if} \\
&(x'- x'')(y'- y'') + (z' - z'')(w' - w'') = (x' - x'')^2 + (z' - z'')^2 \quad \text{ if and only if} \\
&(x' - x'')((y' - x') - (y'' - x'')) + (z' - z'')((w' - z') - (w'' - z'')) = 0.
\end{align*}
Note that after applying the above transformation, $p \mapsto [x', y'-x', z', w'- z']$ and $q \mapsto [x'', y''-x'', z'', w''- z'']$, and these points span the distance $[1, 0]$. By reversing the above chain of equations, we can see that the converse also holds.
That is, $p$ and $q$ span the distance $[1, 2]$ before the transformation if and only if they span the distance $[1,0]$ after the transformation.
We conclude that the repeated distance problem is equivalent for the distances $[1, 2]$ and $[1, 0]$, which completes the proof.

Note that in the above argument, all of the transformations we used preserve the multiplicity (and the secondary multiplicity) of a point set, so we have shown that the repeated distance problem in $\DD^2$ is still equivalent for each non-zero distance $\delta$ even when we allow bounds that are a function of the multiplicity of the set. That is, if there is a set in $\DD^2$ of size $n$ that has multiplicity $n^{\lambda}$ and spans $m$ unit distances, then for any other non-zero distance $\delta$, we can use the above transformations to get another set of size $n$ with multiplicity $n^{\lambda}$ that spans $m$ copies of the distance $\delta$.
\end{proof}

We now move to study the distinct distances problem in $\DD^2$.
\vspace{2mm}

\noindent {\bf Theorem \ref{th:DistinctDual}.}
\emph{Let $\pts$ be a set of $n$ points in $\DD^2$ with secondary multiplicity $n^{\nu}$.
Assume that the points of $\pts$ do not all have the same real part.
Then}
\[ D(\pts) = \Omega \left( n^{1-\nu}\log^{-2} n \right). \]

\begin{proof}
For $0 \leq j \leq \log n$, let $\pts_j$ be the set of points $p\in \pts$ such that the number of points of $\pts$ having the same real part $(x, z)$ as $p$ is at least $2^j$ and smaller than $2^{j + 1}$.
By the pigeonhole principle, there exists $0 \leq j_0 \leq \log n$ such that $|\pts_{j_0}|=\Omega(n/\log n)$.
Fix one such $j_0$ and let $\pts_{\RR} \subseteq \RR^2$ be the set of distinct real parts of the points of $\pts_{j_0}$.
By definition,
\[|\pts_{\RR}| =\Omega\left( \frac{|P_{j_0}|}{2^{j_0}}\right) =\Omega\left( \frac{n}{2^{j_0} \log n}\right).\]

By Theorem \ref{th:GK15 2},
\begin{equation} \label{eq:DDrealDualThm}
D(\pts_{\RR}) = \Omega\left(\frac{n}{2^{j_0} \log n}\cdot\log^{-1}\left(\frac{n}{2^{j_0} \log n}\right)\right) = \Omega\left(\frac{n}{2^{j_0} \log^2 n}\right).
\end{equation}

We first consider the case of $|\pts_{\RR}|=1$, which can only happen when $2^{j_0} =\Omega(n/\log n)$.
In this case, all the points of $\pts_{j_0}$ have the same real part $(x, z)$.
In other words, $\pts_{j_0}$ is contained in a single imaginary plane $H_{(x,z)}$.

By an assumption of the theorem, we may fix a point $p=[x', y', z', w'] \in \pts$ such that $(x', z') \neq (x, z)$.
Consider a point $q=[x, y'', z, w'']$ from $H_{(x,z)}$.
The distance between $p$ and $q$ is
\begin{align}
&\left[\left(\Delta x\right)^2 + \left(\Delta z\right)^2, 2\left(\Delta x \Delta y + \Delta z \Delta w\right)\right] \nonumber \\
&\hspace{15mm}= \left[\left(x - x'\right)^2 + \left(z - z'\right)^2, 2\left(\left(x - x'\right)\left(y'' - y'\right) + \left(z - z'\right)\left(w'' - w'\right)\right)\right]. \label{eq:SecondDistLine}
\end{align}
The first coordinate of \eqref{eq:SecondDistLine} is the same for any choice of $q$ from $H_{(x,z)}$.
The second coordinate of \eqref{eq:SecondDistLine} varies according to the value of $y''(x - x')+w''( z - z')$.

For a fixed $d\in \RR$, consider the set of points $q = [x, y'', z, w'']$ on the imaginary plane $H_{(x,z)}$ that satisfy $y''(x - x')+w''( z - z')=d$.
Since $(x - x', z - z') \neq (0, 0)$, this set forms a line with a direction orthogonal to $(x - x',z - z')$.
By the assumption on the secondary multiplicity, such a line contains at most $n^\nu$ points of $\pts_{j_0}$.
Since $H_{(x,z)}$ contains all $\Omega(n/\log n)$ points of $\pts_{j_0}$, there are $\Omega(n^{1-\nu}/\log n)$ such lines that intersect $\pts_{j_0}$.
In other words, $D(\{p\},\pts_{j_0}) = \Omega \left( n^{1-\nu}\log^{-1} n \right)$, which concludes this case.

We now move to the case where $|\pts_{\RR}|>1$.
Let $\Delta$ be the set of nonzero distances spanned by $\pts_{\RR}$.
Note that, in this case, $|\Delta|\ge 1$.

Fix a distance $d \in \Delta$ and two points $(x_d, z_d), (x'_d, z'_d) \in \pts_{\RR}$ that span $d$.
Then, any points $[x_d, y, z_d, w] \in H_{(x_d, z_d)}$ and $[x'_d, y', z'_d, w'] \in H_{(x'_d, z'_d)}$ span the distance
\begin{equation} \label{eq:dotProdDistinct}
\left[\left(x_d - x'_d\right)^2 + \left(z_d - z'_d\right)^2, 2((x_d - x'_d)(y - y') + (z_d - z_d')(w - w'))\right].
\end{equation}

The first coordinate of \eqref{eq:dotProdDistinct} is $d^2 \neq 0$.
Consider the vector $\vec{v} = (x_d - x_d', z_d - z_d')$.
The second coordinate of \eqref{eq:dotProdDistinct} depends on the expression
\begin{equation} \label{eq:DotProductDiff}
\vec{v} \cdot (y, w) - \vec{v} \cdot (y', w').
\end{equation}

Consider $(y_d, w_d)$ and $(y'_d, w'_d)$ as points in $\RR^2$.
Then, up to a constant factor, $\vec{v} \cdot (y, w)$ can be thought of as the scalar projection of $(y, w)$ along the $\vec{v}$-axis.
That is, the expression \eqref{eq:DotProductDiff} depends only on which lines orthogonal to $\vec{v}$ contain the two points.
In particular, the number of distinct values of \eqref{eq:DotProductDiff} is at least the number of such lines in $H_{(x_d, z_d)}$ that contain at least one point of $\pts_{j_0}$.
By assumption, each such line contains at most $n^\nu$ points.
Since $H_{(x_d, z_d)}$ contains $\Theta(2^{j_0})$ points, there are $\Omega(2^{j_0}/n^\nu)$ such lines.
That is, the expression \eqref{eq:DotProductDiff} attains $\Omega(2^{j_0}/n^\nu)$ distinct values.
This is the number of distinct distances between the points of $\pts$ with real part $(x_d, z_d)$ and the points of $\pts$ with real part $(x'_d, z'_d)$.

In the preceding paragraphs, we considered pairs of points of $\pts$ with fixed real parts spanning a distance $d \in \Delta$.
We concluded that such pairs span $\Omega(2^{j_0}/n^\nu)$ distinct distances.
Recall that the first coordinate of \eqref{eq:dotProdDistinct} is $d^2$.
Therefore, for distinct choices of $d$, this argument produces disjoint sets of distances.
To obtain a lower bound for $D(\pts_{j_0})$, we may thus sum up the number of distinct distances over every $d \in \Delta$.
Combining this with \eqref{eq:DDrealDualThm} gives
\[ D(\pts) \ge D(\pts_{j_0}) = \Omega\left(\frac{n}{2^{j_0}\log^2 n} \cdot \frac{2^{j_0}}{n^{\nu}}\right) = \Omega\left(\frac{n^{1 - \nu}}{\log^2 n}\right). \]
\end{proof}

\section{Double numbers} \label{sec:Double}

We now prove our results in the double plane $\WW^2$.
We restate each result before proving it.

Recall that the notation $\langle d_1,d_2 \rangle$ was defined in \eqref{eq:CoordinateChangeNumber}.
Consider a distance $d = \langle d_1, d_2 \rangle \neq \langle 0, 0 \rangle$.
If both $d_1$ and $d_2$ are positive, then we say that $d$ is a \emph{type A distance}.
If $d_1 = 0$ and $d_2 > 0$, we say that $d$ is a \emph{type B distance}. And if $d_1 > 0$, $d_2 = 0$, we say that $d$ is a \emph{type C distance}.

\vspace{2mm}

\noindent {\bf Theorem \ref{th:UnitDouble}.}
\emph{Let $\pts$ be a set of $n$ points in $\WW^2$. Let $d \in \WW$ be a non-zero distance. \\
(a) If $d$ is of type A and $\pts$ has minimal multiplicity $n^{\lambda}$ (where $0\le \lambda \le 1$), then the number of times $d$ is spanned by $\pts$ is
\[ O \left( \left(n^{4/3 + \lambda/3} + n^{1 + \lambda} \right) \log^2 n\right). \]
(b) If $d$ is of type B and $\pts$ has real multiplicity $n^{\lambda}$, then the number of times $d$ is spanned by $\pts$ is $\displaystyle O \left( n^{1 + \lambda/3}\right)$. \\
(c) If $d$ is of type C and $\pts$ has imaginary multiplicity $n^{\lambda}$, then the number of times $d$ is spanned by $\pts$ is $\displaystyle O \left( n^{1 + \lambda/3}\right)$.}
\begin{proof}
We first show that the repeated distance problem is equivalent for any two distances of the same type.

For a positive $s\in \RR$, consider the transformation $\langle x, y, z, w \rangle \mapsto \langle x\sqrt{s}, y, z\sqrt{s}, w \rangle$.
By observing the distance definition \eqref{eq:DoubleDist}, we note that this bijection takes distances of the form $\langle d_1, d_2 \rangle$ to $\langle s\cdot d_1, d_2 \rangle$.
Similarly, the bijection $\langle x, y, z, w \rangle \mapsto \langle x, y\sqrt{s}, z, w\sqrt{s} \rangle$ takes distances of the form $\langle d_1, d_2 \rangle$ to $\langle d_1, s\cdot d_2 \rangle$.
This implies that the maximum number of repeated distances is the same for any two distances of the same type.

Moreover, we note that because the above transformations do not alter any of the three multiplicities of a point set, the maximum number of copies of a distance in a set with a given cardinality \emph{and} a given minimal multiplicity is also the same for any two type A distances, and similarly for the other two types.
\parag{The type A case.}
By the above, it suffices to prove the bound for the type A distance $d=\langle 1, 1 \rangle$. Recall that the minimal multiplicity was defined as the minimum of the real and imaginary multiplicities. Thus, since the condition $\rho_{\WW^2}(p, q) = d$ is symmetric in the real and imaginary parts of $p$ and $q$, we may assume without loss of generality that $\leq n^{\lambda}$ points of $\pts$ have the same real part.

Let $t$ be the number of pairs of $\pts^2$ that span the distance $d$.
For $0 \leq j \leq \log n$, let $\pts_j$ be the set of points $p \in \pts$ such that the number of points in $\pts$ having the same real part $(x, z)$ as $p$ is at least $2^j$ and smaller than $2^{j+1}$.

By the pigeonhole principle, there exist $0 \leq j_0, j_1 \leq \log n$ such that the number of times $d$ is spanned by pairs of points in $\pts_{j_0}\times \pts_{j_1}$ is $\Omega(t/\log^2n)$.
We will show that this number is also $O\left(n^{4/3 + \lambda/3} + n^{1 + \lambda}\right)$.
Combining these two bounds will then complete the proof of part (a) of the theorem.

Since the distance function defined in \eqref{eq:DoubleDist} is symmetric, we may assume that $j_0 \leq j_1$.
By our assumptions, we have that $2^{j_0} \le 2^{j_1} \le n^{\lambda}$.

Let $\pts_{\RR,0}$ be the set of real parts of the points of $\pts_{j_0}$.
Let $\pts_{\RR,1}$ be the set of real parts of the points of $\pts_{j_1}$.
By the definition of the sets $\pts_{j_0}$ and $\pts_{j_1}$, we have that $|\pts_{\RR,0}| =O\left(n2^{-j_0}\right)$ and $|\pts_{\RR,1}| = O\left( n2^{-j_1}\right)$.
By Theorem \ref{th:Szek93}, the number of unit distances in $\pts_{\RR,0}\times\pts_{\RR,1}$ is
\begin{equation} \label{eq:UnitDoubleInRePlane}
O\left(n^{4/3}2^{-2(j_0+ j_1)/3} + n2^{-j_0}\right).
\end{equation}

Consider two points $(x', z') \in \pts_{\RR,0}$ and $(x'', z'') \in \pts_{\RR,1}$ that span a unit distance in $\R^2$.
Let $S \subset \R^2$ be the set of points $(y, w) \in \R^2$ such that $\langle x', y, z', w \rangle$ is in $\pts_{j_0}$.
Let $T \subset \R^2$ be the set of points $(y, w) \in \R^2$ such that $\langle x'', y, z'', w \rangle$ is in $\pts_{j_1}$.
By definition, $|S| =\Theta(2^{j_0})$ and $|T| =\Theta(2^{j_1})$.
By Theorem \ref{th:Szek93}, the number of unit distances in $S\times T$ is
\begin{equation} \label{eq:UnitDoubleInImPlane}
O\left(2^{2(j_0+j_1)/3} + 2^{j_1}\right).
\end{equation}
In other words, this is an upper bound on the number of pairs $(p,q)\in \pts_{j_0}\times\pts_{j_1}$ such the real part of $p$ is $(x',z')$, the real part of $q$ is $(x'',z'')$, and the distance between $p$ and $q$ is $\langle 1, 1 \rangle$.

Recall the bound in \eqref{eq:UnitDoubleInRePlane} on the number of pairs of $\pts_{\RR,0}\times \pts_{\RR,1}$ that span a unit distance.
By summing \eqref{eq:UnitDoubleInImPlane} over each of these pairs, we obtain that the number of pairs in $\pts_{j_0}\times \pts_{j_1}$ that span the distance $\langle 1, 1 \rangle$ is
\begin{align*}
&O\left(n^{4/3}2^{-2(j_0+ j_1)/3} + n2^{-j_0}\right) \cdot O\left(2^{2(j_0+j_1)/3} + 2^{j_1}\right) \\
&\hspace{10mm}=O\left(n^{4/3} + n^{4/3}2^{(j_1-2j_0)/3} + n2^{(2j_1-j_0)/3} + n2^{j_1 - j_0}\right) =O\left(n^{4/3 +1/3 \cdot\lambda} + n^{1 + \lambda}\right).
\end{align*}
This completes the proof of part (a) of the theorem.

\parag{The type B case.}
By the above, it suffices to prove the bound for the type B distance $\langle 0, 1 \rangle$.
Note that this distance is spanned only by points having the same real part.
Thus, it suffices to separately consider every imaginary plane $H_{(x,z)}$.
Set $\pts_{(x,z)} = \pts \cap H_{(x,z)}$.
By assumption, for every real part $(x,z)$, we have that $|\pts_{(x, z)}|\le n^{\lambda}$.
We also have that $\sum |\pts_{(x, z)}| = n$.

By Theorem \ref{th:Szek93}, the number of $\langle 0, 1 \rangle$-distances in $H_{(x,z)}$ is $O(|\pts_{(x, z)}|^{4/3})$. We conclude that the number of pairs of $\pts^2$ that span the distance $\langle 0, 1 \rangle$ is
\[\sum_{(x,z)} O\left(|\pts_{(x, z)}|^{4/3}\right) =O\left( n^{\lambda/3} \sum_{(x,z)} |\pts_{(x, z)}|\right) = O\left(n^{1 + \lambda/3}\right).\]

\parag{The type C case.}
This follows from the type B case, by symmetry.
\end{proof}

We now consider several constructions that provide lower bounds for the repeated distance problem in $\WW^2$.

\begin{claim}
(a) Let $d$ be a type A distance. There exists a set $\pts\subset\WW^2$ of size $n$ such that the number of pairs in $\pts^2$ that span $d$ is $\Theta(n^2)$.\\[2mm]
(b) Assuming that the unit distance bound of Theorem \ref{th:Szek93} is sharp when $m=n^{1-\lambda}$, the bound of part (a) of Theorem \ref{th:UnitDouble} is also sharp for sets of size $n$ and minimal multiplicity $n^{\lambda}$, up to polylogarithmic factors.\\[2mm]
(c) Assuming that the unit distance bound of Theorem \ref{th:Szek93} is sharp when $m = n$, the bound of part (b) (part (c)) of Theorem \ref{th:UnitDouble} is also sharp for sets of size $n$ and real (imaginary) multiplicity $n^{\lambda}$.
\end{claim}
\begin{proof}
(a) We prove the claim for $d=\langle 1, 1 \rangle$.
By the discussion at the beginning of the proof of Theorem \ref{th:UnitDouble}, this extends to any type $A$ distance.

We adapt Lenz's classical unit distance construction (see \cite{erd46}).
Let $C\subset \RR^2$ be the unit circle centered at the origin.
Let $\pts_1$ be a set of $n$ points of the form $\langle x,0,z, 0 \rangle$, such that $(x,z)\in \RR^2$ is incident to $C$.
Let $\pts_2$ be a set of $n$ points of the form $\langle 0,y,0,w \rangle$, such that $(y,w)\in \RR^2$ is incident to $C$.
Set $\pts = \pts_1 \cup \pts_2$.
Then the distance between any pair of points in $\pts_1\times \pts_2$ is
\[\langle x^2 + z^2, y^2 + w^2 \rangle = \langle 1, 1 \rangle.\]
The claim follows since the cardinality of the set is $\Theta(n)$ and there are $\Theta(n^2)$ pairs in $\pts_1\times \pts_2$.

(b) It again suffices to prove the claim for the distance $\langle 1, 1 \rangle$.
By the assumption, there exist sets $\pts_1,\pts_2 \subset \RR^2$ such that $|\pts_1|=n$, $|\pts_2|=n^{1-\lambda}$, and the number of unit distances in $\pts_1\times\pts_2$ is $\Theta(n^{(4-2\lambda)/3}+n)$.
Let $C\subset \RR^2$ be the unit circle centered at the origin and let $\pts_C$ be a set of $n^\lambda$ arbitrary points incident to $C$.

Let $\pts_1'$ be the set of points $\langle x, 0, z, 0 \rangle$ such that $(x,z)\in \pts_1$.
By definition, $\pts'_1$ is a set of $n$ points with real multiplicity one.
Let $\pts_2'$ be the set of points $\langle x, y, z, w \rangle$ such that $(x,z)\in \pts_2$ and $(y,w)\in\pts_C$.
By definition, $\pts'_2$ is a set of $n$ points with real multiplicity $n^\lambda$.
We set $\pts = \pts'_1 \cup \pts'_2$.
We see that $|\pts|=\Theta(n)$ and the real multiplicity of $\pts$ is $\Theta(n^\lambda)$, which means that the minimal multiplicity is $O(n^\lambda)$ (in fact, it is $\Theta(n^\lambda)$, because the imaginary multiplicity is $\Theta(n)$).
The number of pairs of $\pts'_1 \cup \pts'_2$ that span the distance $\langle 1, 1 \rangle$ is
\[ \Theta(n^{(4-2\lambda)/3}+n) \cdot \Theta(n^{\lambda}) = \Theta\left(n^{4/3+\lambda/3}+n^{1+\lambda}\right).\]

(c) We prove the claim for part (b) of Theorem \ref{th:UnitDouble}, for the distance $\langle 0, 1 \rangle$.
By the discussion at the beginning of the proof of Theorem \ref{th:UnitDouble}, this extends to any type B distance. 
By the assumption, we may choose a set $\pts_{\RR} \subset \RR^2$ of $n^\lambda$ points that determines $\Theta(n^{4\lambda/3})$ unit distances.
In addition, define $\pts'_{\RR} \subset \RR^2$ to be an arbitrary set of $n^{1-\lambda}$ points.

Let $\pts$ be the set of points $\langle x, y, z, w \rangle$ such that $(x,z)\in \pts'_{\RR}$ and $(y,w)\in \pts_{\RR}$.
Then $\pts$ is a set of $n$ points with real multiplicity $n^\lambda$.
For a pair of $\pts^2$ to span the distance $\langle 0, 1 \rangle$, the two points must have the same real part and their imaginary parts must define a unit distance in $\RR^2$.
Thus, the number of pairs of $\pts^2$ that span $\langle 0, 1 \rangle$ is
\[ \Theta\left(n^{1-\lambda}\right)\cdot \Theta\left(n^{4\lambda/3}\right) = \Theta\left(n^{1+\lambda/3}\right). \]

By symmetry, this also proves the desired sharpness result for part (c) of Theorem \ref{th:UnitDouble}.
\end{proof}

We now move to the distinct distances problem in $\WW^2$.
\vspace{2mm}

\noindent {\bf Theorem \ref{th:DistinctDouble}}
\emph{Every set of $n$ points in $\WW^2$ determines $\Omega \left( \frac{n}{\log^3 n} \right)$ distinct distances.}

\begin{proof}
We repeat several parts of the proof of Theorem \ref{th:DistinctDual}.
For $0 \leq j \leq \log n$, let $\pts_j$ be the set of points $p\in \pts$ such that the number of points of $\pts$ having the same real part $(x, z)$ as $p$ is at least $2^j$ and smaller than $2^{j + 1}$.
By the pigeonhole principle, there exists $0 \leq j_0 \leq \log n$ such that $|\pts_{j_0}|=\Omega(n/\log n)$.
Fix one such $j_0$ and let $\pts_{\RR} \subseteq \RR^2$ be the set of distinct real parts of the points of $\pts_{j_0}$.
By definition,
\[|\pts_{\RR}| =\Omega\left( \frac{|P_{j_0}|}{2^{j_0}}\right) =\Omega\left( \frac{n}{2^{j_0} \log n}\right).\]

By Theorem \ref{th:GK15 2},
\begin{equation} \label{eq:DDrealDoubleThm}
D(\pts_{\RR}) = \Omega\left(\frac{n}{2^{j_0} \log n}\cdot\log^{-1}\left(\frac{n}{2^{j_0} \log n}\right)\right) = \Omega\left(\frac{n}{2^{j_0} \log^2 n}\right).
\end{equation}

Let $\Delta$ be the set of distances spanned by $\pts_{\RR}$ (note that $\Delta$ may consist only of 0, if all points of $\pts_{j_0}$ have the same real part).
Fix a distance $d \in \Delta$ and a pair of points $(x_d, z_d), (x'_d, z'_d) \in \pts_{\RR}$ that span $d$.
Let $S \subset \R^2$ be the set of points $(y, w) \in \R^2$ such that $\langle x_d, y, z_d, w \rangle $ is in $\pts_{j_0}$.
Let $T \subset \R^2$ be the set of points $(y, w) \in \R^2$ such that $\langle x'_d, y, z'_d, w \rangle $ is in $\pts_{j_0}$.

By definition, we have that $|S|= \Theta(2^{j_0})$ and $|T|= \Theta(2^{j_0})$.
By Theorem \ref{th:M19}, we have that
\[D(S,T) = \Omega\left( \frac{\sqrt{2^{j_0}\cdot 2^{j_0}} }{j_0}\right) = \Omega\left(\frac{2^{j_0}}{j_0}\right). \]

From the distance definition \eqref{eq:DoubleDist}, we see that the sets of bipartite distances between $S$ and $T$ that we get from different choices of $d \in \Delta$ are disjoint. Thus, to obtain a lower bound for $D(\pts)$, we may sum the above bound over every $d\in \Delta$.
By recalling \eqref{eq:DDrealDoubleThm} and that $j_0\le \log n$, we obtain
\[ D(\pts) =  \Omega\left(\frac{n}{2^{j_0} \log^2 n}\right) \cdot \Omega\left(\frac{2^{j_0}}{j_0}\right)  = \Omega\left(\frac{n}{ \log^3 n}\right). \]
\end{proof}

Finally, we show an upper bound that matches the lower bound of Theorem \ref{th:DistinctDouble}, up to a factor of $\log^2 n$.

\begin{claim}
There exists a set $\pts$ of $n$ points in $\WW^2$, such that $D(\pts)=O \left( \frac{n}{\log n} \right)$.
\end{claim}
\begin{proof}
Let $\pts'\subset \RR^2$ be a set of $n^{1/2}$ points, such that $D(\pts') = \Theta(n^{1/2}/\sqrt{\log n})$. For example, we could take $\pts$ to be an $n^{\frac 1 4} \times n^{\frac 1 4}$ square grid \cite{erd46}.
Set
\[ \pts = \left\{\langle x, y, z, w\rangle \in \WW^2 :\ (x, z) \in \pts' \quad \text{ and } \quad (y, w) \in \pts' \right\}. \]

Note that $\pts$ is a set of $n$ points in $\WW^2$.
From the distance definition (\ref{eq:DoubleDist}), we get that
\[ D(\pts)= \Theta(n^{1/2}/\sqrt{\log n}) \cdot \Theta(n^{1/2}/\sqrt{\log n}) =\Theta(n/\log n). \]
\end{proof}

\section*{Acknowledgments}
First and foremost, the author wishes to thank his research mentor Adam Sheffer, who originally posed the problems studied here, served as an essential resource throughout the summer, and provided copious feedback on the many drafts of this paper. He would also like to thank everyone else involved with the 2019 CUNY REU for their many helpful conversations and suggestions, especially Surya Mathialagan, on whose results he relied in Theorem \ref{th:DistinctDouble}. He also wishes to thank Tamas Fleiner, Alex Iosevich, Benjamin Lund, and Victoria Talvola. .

\end{document}